\documentclass[11pt,a4paper]{amsart}
\usepackage{amsthm,amsfonts,amsmath, amssymb,graphicx,color}
\usepackage{iwona}
\usepackage[numbers,sort&compress]{natbib}
\usepackage[lmargin=33mm,rmargin=33mm,tmargin=33mm,bmargin=33mm]{geometry}
\usepackage[colorlinks=true,citecolor=black,linkcolor=black,urlcolor=blue]{hyperref}
\renewcommand{\baselinestretch}{1.16}
\allowdisplaybreaks

\renewcommand{\geq}{\geqslant}
\renewcommand{\leq}{\leqslant}

\newcommand{\doi}[1]{doi:\,\href{http://dx.doi.org/#1}{#1}}

\theoremstyle{plain}
\newtheorem{theorem}{Theorem}
\newtheorem{lemma}[theorem]{Lemma}
\newtheorem{corollary}[theorem]{Corollary}
\newtheorem{proposition}[theorem]{Proposition}
\newtheorem{observation}[theorem]{Observation}
\theoremstyle{definition}
\newtheorem{conjecture}[theorem]{Conjecture}

\DeclareMathOperator{\CR}{cr}
\newcommand{\eps}{\epsilon}  

\newcommand{\floor}[1]{\ensuremath{\protect\lfloor{#1}\rfloor}}

\newcommand{\mededges}{E_M} 
\newcommand{\smallpairs}{P_S} 
\newcommand{\medpairs}{P_M} 
\newcommand{\largepairs}{P_L} 
\newcommand{\smallinc}{I_S} 
\newcommand{\medinc}{I_M} 
\newcommand{\smalllines}{L_S} 

\begin{document}

\title{Progress on Dirac's Conjecture}

\author[]{Michael~S.~Payne}
\address{
\newline Department of Mathematics and Statistics 
\newline The University of Melbourne
\newline Melbourne, Australia}
\email{m.payne3@pgrad.unimelb.edu.au}

\author[]{David~R.~Wood}
\address{
\newline School of Mathematical Sciences
\newline Monash University
\newline Melbourne, Australia}
\email{david.wood@monash.edu}

\thanks{Michael Payne is supported by an Australian Postgraduate Award from the Australian Government. 
Research of David Wood is supported by the Australian Research Council.}

\subjclass[2000]{}

\date{\today}

\begin{abstract}In 1951, Gabriel Dirac conjectured that every set $P$
  of $n$ non-collinear points in the plane contains a point in at
  least $\frac{n}{2}-c$ lines determined by $P$, for some constant
  $c$. The following weakening was proved by Beck and
  Szemer\'edi--Trotter: every set $P$ of $n$ non-collinear points contains a
  point in at least $\frac{n}{c}$ lines determined by $P$, for some
  large unspecified constant $c$. We prove that every set $P$ of $n$
  non-collinear points contains a point in at least $\frac{n}{37}$
  lines determined by $P$. 
We also give the best known constant for Beck's Theorem, proving that every set of $n$ points with at most $\ell$ collinear determines at least $\frac{1}{98} n(n-\ell)$ lines.
\end{abstract}

\maketitle

\section{Introduction}

Let $P$ be a finite set of points in the plane. A line that contains
at least two points in $P$ is said to be \emph{determined} by $P$. In
1951, \citet{Dirac51} 
made the following conjecture, which remains
unresolved:

\begin{conjecture}[Dirac's Conjecture]
  Every set $P$ of $n$ non-collinear points contains a point in at
  least $\frac{n}{2}-c_1$ lines determined by $P$, for some constant
  $c_1$.
\end{conjecture}

See reference \citep{AIKN-GC11} for examples showing that the
$\frac{n}{2}$ bound would be tight. Note that if $P$ is non-collinear and contains at least $\frac{n}{2}$ collinear points, then Dirac's Conjecture holds. Thus we may assume that $P$ contains at most $\frac{n}{2}$ collinear points, and $n\geq 5$.
In 1961, \citet{Erdos61} proposed
the following weakened conjecture.

\begin{conjecture}[Weak Dirac Conjecture]
  Every set $P$ of $n$ non-collinear points contains a point in at
  least $\frac{n}{c_2}$ lines determined by $P$, for some constant $c_2$.
\end{conjecture}

In 1983, the Weak Dirac Conjecture was proved indepedently by
\citet{Beck-Comb83} and \citet{SzemTrot-Comb83}, in both cases with $c_2$
unspecified and very large. 
We prove the Weak Dirac Conjecture with
$c_2$ much smaller. (See references
\citep{ErdosPurdy95,ErdosPurdy78,KellyMoser58,Purdy81,LPS12} for more
on Dirac's Conjecture.)

\begin{theorem}
  \label{Main}
  Every set $P$ of $n$ non-collinear points contains a point in at
  least $\frac{n}{37}$ lines determined by $P$.
\end{theorem}

Theorem~\ref{Main} is a consequence of the following theorem. The points of $P$ together with the lines determined by $P$ are called the \emph{arrangement} 
 of $P$.

\begin{theorem}
  \label{NumEdges}
  For every set $P$ of $n$ points in the plane with at most
  $\frac{n}{37}$ collinear points, the arrangement of $P$ has at
  least $\tfrac{n^2}{37}$ point-line incidences.
\end{theorem}

\begin{proof}[Proof of Theorem~\ref{Main} assuming Theorem~\ref{NumEdges}] 
  Let $P$ be a set of $n$ non-collinear points in the plane.  If $P$
  contains at least $\frac{n}{37}$ collinear points, then every other
  point is in at least $\frac{n}{37}$ lines determined by $P$ (one
  through each of the collinear points).
  Otherwise, by Theorem~\ref{NumEdges}, the arrangement of
  $P$ has at least $\tfrac{n^2}{37}$ incidences, and so some point is incident with at least $\frac{n}{37}$ lines determined by $P$.
\end{proof}

In his work on the Weak Dirac Conjecture, Beck proved the following theorem~\cite{Beck-Comb83}.

\begin{theorem}[Beck's Theorem]\label{beckthm}
Every set $P$ of $n$ points with at most $\ell$ collinear determines at least $c_3 n(n-\ell)$ lines, for some constant $c_3$.
\end{theorem}

In Section~\ref{acfbt} we use the proof of Theorem~\ref{NumEdges} and some simple lemmas to show that $c_3 \geq \frac{1}{98}$. Similar methods and a bit more effort yield $c_3 \geq \frac{1}{93}$ (see~\cite{thesis} for details).

\section{Proof of Theorem~\ref{NumEdges}}
\label{tdc}

The proof of Theorem~\ref{NumEdges} takes inspiration from
the well known proof of Beck's Theorem \citep{BecksTheorem} as a
corollary of the Szemer\'edi--Trotter Theorem \citep{SzemTrot-Comb83},
and also from the simple proof of the Szemer\'edi--Trotter Theorem due to
\citet{Szekely-CPC97}, which in turn is based on the Crossing Lemma. 

The \emph{crossing number} of a graph $G$, denoted by $\CR(G)$, is the
minimum number of crossings in a drawing of $G$. 
The following lower bound on $\CR(G)$ was first proved by
\citet{Ajtai82} and \citet{Leighton83} (with worse constants).  
A simple proof with better constants can be found in \citep{Proofs3}.  
The following version is due to \citet{PRTT-DCG06}.

\begin{theorem}[Crossing Lemma]
  \label{CrossingLemma} For every graph $G$ with $n$ vertices and $m\geq\frac{103}{16}n$
  edges, $$\CR(G)\geq\frac{1024\,m^3}{31827\,n^2} \enspace. $$ 
\end{theorem}

In fact, we employ a slight strengthening of the Szemer\'edi--Trotter Theorem formulated in terms of visibility graphs.
The \emph{visibility graph} $G$ of a point set $P$ has vertex set $P$,
where $vw\in E(G)$ whenever the line segment $vw$ contains no other
point in $P$ (that is, $v$ and $w$ are consecutive on a line
determined by $P$).

For $i \geq 2$, an  \emph{$i$-line} is a line containing exactly $i$ points  in $P$. 
Let $s_i$ be the number of $i$-lines. 
Let $G_i$ be the spanning subgraph of the visibility
graph of $P$ consisting of all edges in $j$-lines where $j\geq
i$; see Figure~\ref{Grid} for an example.  
Note that since each $i$-line contributes $i-1$ edges, $|E(G_i)|=\sum_{j\geq i}(j-1)s_j$. 
Part (a) of the following version of the Szemer\'edi--Trotter Theorem gives a bound on $|E(G_i)|$, while part (b) is the well known version that bounds the number of $j$-lines for $j\geq i$.

\begin{figure}
  \includegraphics[width=\textwidth]{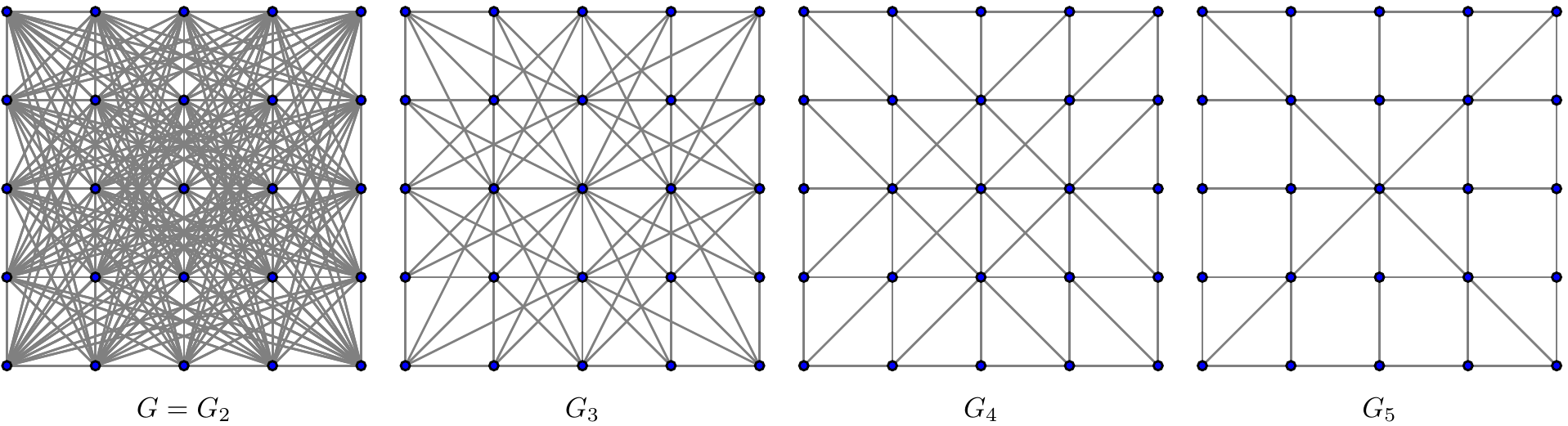}
  \caption{The graphs $G_2,G_3,G_4,G_5$ in the case of the $5\times 5$
    grid.}
  \label{Grid}
\end{figure}

\begin{theorem}[Szemer\'edi--Trotter Theorem] 
\label{szemtrot}
 Let $\alpha$ and $\beta$ be positive constants such that every graph
  $H$ with $n$ vertices and $m\geq\alpha n$ edges
  satisfies $$\CR(H)\geq \frac{m^3}{\beta n^2}\enspace.$$
Let $P$ be a set of $n$ points in the plane.  
Then 
\begin{alignat*}{3}
\mbox{(a)} \enspace && \enspace \sum_{j\geq    i}(j-1)s_j
 &\leq \max \left\{ \alpha n,  \frac{\beta \,n^2 }{2(i-1)^2} \right\} \enspace , \\
\mbox{and (b)} \enspace && \sum_{j\geq i}s_j &\leq \max\left\{ \frac{\alpha n}{i-1},  \frac{\beta \,n^2 }{2(i-1)^3} \right\} \enspace .
\end{alignat*}
\end{theorem}

\begin{proof}
Suppose $\sum_{j\geq    i}(j-1)s_j =|E(G_i)| \geq \alpha n$. Then by the assumed
Crossing Lemma applied to $G_i$,
$$
\CR(G_i) 
\geq \frac{|E(G_i)|^3}{\beta n^2} 
=\frac{(\sum_{j\geq    i}(j-1)s_j)^2|E(G_i)|}{\beta n^2} 
\geq \frac{(i-1)^2(\sum_{j \geq    i}s_j)^2|E(G_i)|}{\beta n^2} 
\enspace.
$$
On the other hand, since two lines cross at most once,
$$
\CR(G_i)
\leq\binom{\sum_{j\geq i} s_j}{2} \leq\frac{1}{2}\Big(\sum_{j\geq
  i}s_j\Big)^2 \enspace.
$$ Combining these inequalities yields part (a). Part (b) follows directly from part (a).
\end{proof}

The proof of Theorem~\ref{NumEdges} also employs Hirzebruch's Inequality~\cite{Hirzebruch}.

\begin{theorem}[Hirzebruch's Inequality]
Let $P$ be a set of $n$ points with at most $n-3$ collinear. 
Then $$ s_2 +\frac{3}{4}s_3 \geq n + \sum_{i\geq5}(2i-9)s_i \enspace  . $$
\end{theorem}

Theorem~\ref{NumEdges} follows from Theorem~\ref{CrossingLemma} and
the following general result by setting 
$\alpha=\frac{103}{16}$, $\beta=\frac{31827}{1024}$, $c=71$, and $\delta = \eps$, in which case $\delta \geq \frac{1}{36.158}$. 
The value of $\delta$ is readily calculated numerically since 
      since 
$$\sum_{i\geq c}\frac{i+1}{i^3}
    = \sum_{i\geq 1}\frac{i+1}{i^3} - \sum_{i=1}^{c-1}\frac{i+1}{i^3} 
    = \zeta(2) + \zeta(3)  -  \sum_{i=1}^{c-1}\frac{i+1}{i^3} 
    = 2.847\ldots -  \sum_{i=1}^{c-1}\frac{i+1}{i^3}\enspace, $$
   where $\zeta$ is the Riemann zeta function.

\begin{theorem}
  \label{NumEdgesGeneral}
  Let $\alpha$ and $\beta$ be positive constants such that every graph
  $H$ with $n$ vertices and $m\geq\alpha n$ edges
  satisfies $$\CR(H)\geq \frac{m^3}{\beta n^2}\enspace.$$ 
 Fix an integer $c\geq 8$ and a real $\eps\in(0,\frac{1}{2})$. Let $h:= \frac{c(c-2)}{5c-18}$. 
Then for every set $P$ of $n$
points in the plane with at most $\eps n$ collinear points, the
arrangement of $P$ has at least $\delta n^2$ point-line incidences, where 
$$
\delta = \frac{1}{h+1} 
\left( 1-\eps\alpha - 
 \frac{\beta}{2} \left(\frac{(c-h-2)(c+1)}{c^3} + \sum_{i\geq c}\frac{i+1}{i^3} \right)  \right)  
\enspace .
$$
\end{theorem}

\begin{proof} 
Let $J:=\{2,3,\dots,\floor{\eps n}\}$. 
Considering the visibility graph $G$ of $P$ and its subgraphs $G_i$ as defined previously, let $k$ be the minimum integer such that
$|E(G_k)|\leq\alpha n$. If there is no such $k$ then let $k:=\floor{\epsilon n}+1$. 
An integer $i\in J$ is \emph{large} if $i\geq k$, and is \emph{small}
if $i\leq c$. An integer in $J$ that is neither large nor small is
\emph{medium}.
 
An \emph{$i$-pair} is
  a pair of points in an $i$-line.  
 A \emph{small pair} is an $i$-pair for some small $i$.
Define \emph{medium pairs} and \emph{large pairs} analogously, and let $\smallpairs, \medpairs$ and $\largepairs$ denote the number of small, medium and large pairs respectively. 
An \emph{$i$-incidence} is an incidence between a point of $P$ and an $i$-line.
A \emph{small incidence} is an $i$-incidence for some small $i$.  Define
\emph{medium incidences} 
analogously, and let $\smallinc$ and $\medinc$ denote the number of small and medium incidences respectively. Let $I$ denote the total number of incidences. Thus, $$I = \sum_{i \in J} is_i \enspace. $$

The proof procedes by establishing an upper bound on the number of small
pairs in terms of the number of small incidences. Analogous bounds are
proved for the number of medium pairs, and the number of large
pairs. Combining these results gives the desired lower bound on the
total number of incidences.

For the bound on small pairs, Hirzebruch's Inequality is useful. 
Since at most $\frac{n}{2}$ points are collinear and $n\geq 5$, there are no more than $n-3$ collinear points.
Therefore, Hirzebruch's Inequality implies that $ hs_2 +\frac{3h}{4}s_3 - hn - h\sum_{i\geq5}(2i-9)s_i \geq 0 $ since $h>0$. 
Thus,
\begin{align*}
   \smallpairs  
  =\;& s_2 + 3s_3 + 6s_4 + \sum_{i=5}^c\binom{i}{2}s_i \\
  \leq\; &  (h+1)s_2 + \left(\frac{3h}{4} +3\right)s_3 + 6s_4 + \sum_{i=5}^c \binom{i}{2}s_i - hn -h\sum_{i=5}^c (2i-9)s_i  \\ 
  \leq\;  & \frac{h+1}{2}\cdot 2s_2 + \frac{h+4}{4}\cdot 3s_3 + \frac{3}{2}\cdot 4s_4 + \sum_{i=5}^c \left(\frac{i-1}{2} - 2h +\frac{9h}{i} \right)is_i -hn \enspace.
\end{align*}
Setting
$X: = \max
 \left\{ \frac{h+1}{2}, \frac{h+4}{4}, \frac{3}{2}, \max_{5\leq i\leq c} \left(\frac{i-1}{2} - 2h +\frac{9h}{i} \right) 
 \right\} 
$
%
%
implies that 
\begin{equation}
  \label{SmallPairs}
  \smallpairs \leq X \smallinc -hn 
 \enspace.
\end{equation}
Considering the second partial derivative with respect to $i$ shows that $\frac{i-1}{2} - 2h +\frac{9h}{i}$ is maximised for $i=5$ or $i=c$. Some linear optimisation shows that, since $c\geq 8$, $X$ is minimised when $h= \frac{c(c-2)}{5c-18}$ and $X= \frac{h+1}{2} = \frac{c-1}{2} - 2h +\frac{9h}{c} $.

To bound the number of medium pairs, consider a medium $i\in
J$. Since $i$ is not large, $\sum_{j\geq i} (j-1)s_j >\alpha n$.
Hence, using parts (a) and (b) of the Szemer\'edi--Trotter Theorem, 
\begin{equation}\label{MediumLines}
\sum_{j\geq i} j s_j
=
 \sum_{j\geq i} (j-1) s_j  +
 \sum_{j\geq i} s_j 
\leq
\frac{\beta n^2}{2(i-1)^2}
+ \frac{\beta n^2}{2(i-1)^3}
=
\frac{\beta n^2 i}{2(i-1)^3} \enspace.  
\end{equation}


Given the factor $X$ in the bound on the number of small pairs in \eqref{SmallPairs}, it helps to introduce the same factor in the bound on the number of medium pairs. It will be convenient to define $Y:=c-1-2X$.
\begin{align*}
   \medpairs - X\medinc 
  =\; & \left( \sum_{i=c+1}^{k-1} \binom{i}{2}s_i \right) - 
  X\left(\sum_{i=c+1}^{k-1} is_i\right)  \\
  =\; &
  \frac{1}{2}\sum_{i=c+1}^{k-1} \left(i-1 - 2X \right)is_i  \\
  =\; &
  \frac{1}{2}\sum_{i=c+1}^{k-1} \left(i-c + Y \right)is_i  \\
  =\;  &
 \frac{1}{2}\left( \sum_{i= c+1}^{k-1}\sum_{j= i}^{k-1}js_j \right) + \frac{Y}{2} \left( \sum_{i=c+1}^{k-1} is_i \right) \enspace .  \\
\end{align*}
Applying \eqref{MediumLines} yields
%
\begin{equation}  \medpairs - X\medinc \leq 
\frac{\beta \,n^2}{4} \left(Y\frac{c+1}{c^3} + \sum_{i\geq c}\frac{i+1}{i^3} \right) \label{MediumPairs} \enspace.
\end{equation}

It remains to bound the number of large pairs:
\begin{equation}
  \label{LargePairs}
\largepairs
  =  \sum_{i=k}^{\floor{\eps n}} \binom{i}{2}s_i
  \leq  \frac{\eps n}{2}\sum_{i\geq k}(i-1)s_i
  =  \frac{\eps n}{2} |E(G_k)|
  \leq  \frac{\eps\alpha \,n^2}{2} \enspace.
\end{equation}
Combining \eqref{SmallPairs}, \eqref{MediumPairs} and \eqref{LargePairs},
\begin{align*}
  \binom{n}{2} =
  \frac{1}{2}(n^2-n) \leq\; & \smallpairs +\medpairs + \largepairs\\
  \leq\; &
  X\smallinc -h n + X\medinc  + 
 \frac{\beta \,n^2}{4} \left(Y\frac{c+1}{c^3} + \sum_{i\geq c}\frac{i+1}{i^3} \right) + \frac{\eps\alpha \,n^2}{2}
    \enspace.
\end{align*}
Thus,
$$ 
I \geq \smallinc + \medinc \geq  \frac{1}{2X} 
\left( 1-\eps\alpha - 
 \frac{\beta}{2} \left(Y\frac{c+1}{c^3} + \sum_{i\geq c}\frac{i+1}{i^3} \right)  \right) n^2 +\frac{2h-1}{2X} n
 \enspace.
$$ 
The result follows since $h\geq 1$.
%
%
\end{proof}

\section{A constant for Beck's Theorem}
\label{acfbt}

Beck proved Theorem~\ref{beckthm} as part of his work on Dirac's Conjecture~\cite{Beck-Comb83}. 
Theorem~\ref{NumEdgesGeneral} from the previous section and Lemmas~\ref{kmprop} and~\ref{Extra} below can be used to give the best known constant in Beck's Theorem. 
\begin{theorem}\label{beckthmc}
Every set $P$ of $n$ points with at most $\ell$ collinear determines at least $\frac{1}{98}n(n-\ell)$ lines.
\end{theorem}

The following lemma, due to Kelly and Moser~\cite{KellyMoser58}, follows directly from Melchior's Inequality~\cite{Melchior-DM41}, which states
that $s_2 \geq 3+ \sum_{i\geq 4}(i-3)s_i$. 
As before, $I$ is the total number of incidences in the arrangement of $P$. Let $E$ be the total number of edges in the visibility graph of $P$, and let $L$ be the total number of lines in the arrangement of $P$.

\begin{lemma}[Kelly--Moser]\label{kmprop} If $P$ is not collinear, then $3L \geq 3 + I$, and since $I = E +L$, also $2L \geq 3 + E$.
\end{lemma}


When there is a large number of collinear points, the following lemma becomes stronger than Theorem~\ref{NumEdgesGeneral}.

\begin{lemma}
  \label{Extra}
  Let $P$ be a set of $n$ points in the plane such that some line 
  contains exactly $\ell$ points in $P$.  Then the visibility graph of $P$ contains at least $\ell(n-\ell)$ edges.
\end{lemma}

\begin{proof} Let $S$ be the set of $\ell$ collinear points in $P$.
  For each point $v\in S$ and for each point $w\in P \setminus S$, count
  the edge incident to $w$ in the direction of $v$.  Since $S$
  is collinear and $w$ is not in $S$, no edge is counted twice. Thus
  $E\geq |S|\cdot|P\setminus S|=\ell(n-\ell)$.
\end{proof}

\begin{proof}[Proof of Theorem~\ref{beckthmc}]
Assume $\ell$ is the size of the largest collinear subset of $P$. If $\ell \geq \frac{n}{49}$ then $E \geq \frac{1}{49}n(n-\ell)$ by Lemma~\ref{Extra} and thus $L > \frac{1}{98}n(n-\ell)$ by Lemma~\ref{kmprop}.
On the other hand, suppose $\ell \leq \frac{n}{49}$. Setting $\frac{\epsilon}{2} = \frac{\delta}{3}$ and  $c=67$ in Theorem~\ref{NumEdgesGeneral} gives $\epsilon \leq \frac{1}{49}$ and $\delta \geq \frac{1}{32.57}$.
So $I \geq \frac{1}{32.57}n^2 \geq \frac{1}{32.57}n(n-\ell)$ 
and thus $L >  \frac{1}{98}n(n-\ell)$ by Lemma~\ref{kmprop}. 
\end{proof}

A more direct approach similar to the methods used in the proof of Theorem~\ref{NumEdgesGeneral} can be shown to improve Theorem~\ref{beckthmc} slightly to yield $\frac{1}{93}n(n-\ell)$ lines. The details are omitted, but can be found in~\cite{thesis}.

Beck's Theorem is often stated as a bound on the number of lines with few points.
In his original paper Beck~\cite{Beck-Comb83} mentioned briefly in a footnote that Lemma~\ref{kmprop} implies the following. 

\begin{observation}[Beck] \label{beckobs}
If $P$ is not collinear, then at least half the lines determined by $P$ contain $3$ points or less.
\end{observation}

\begin{proof} 
By Lemma~\ref{kmprop},
$$ 3s_2 + 3s_3 + 3 \sum_{i\geq 4} s_i > \sum_{i\geq 2} is_i \geq 2s_2
+ 2s_3 + 4\sum_{i\geq 4} s_i\enspace.$$
Thus
$$ 2(s_2 + s_3) > \sum_{i\geq 2} s_i\enspace,$$
as desired.
%
\end{proof}

\begin{corollary} Every set $P$ of $n$ points with at most $\ell$ collinear determines at least $\frac{1}{196}n(n-\ell)$ lines each with at most $3$ points.
\end{corollary}

\bibliography{myBibliography,myConferences}
\bibliographystyle{myNatbibStyle}
\end{document}